\documentclass[11pt,a4paper]{article}
\usepackage[margin=2.3cm]{geometry}
\usepackage{amsmath,amssymb,amsthm}
\usepackage{booktabs,graphicx}
\usepackage[colorlinks=true,linkcolor=blue,citecolor=blue,urlcolor=blue]{hyperref}
\usepackage{enumitem}
\theoremstyle{plain}
\newtheorem{theorem}{Theorem}
\newtheorem{proposition}[theorem]{Proposition}
\newtheorem{lemma}[theorem]{Lemma}
\newtheorem{corollary}[theorem]{Corollary}
\newtheorem{conjecture}[theorem]{Conjecture}
\DeclareMathOperator{\Aut}{Aut}
\theoremstyle{definition}
\newtheorem{definition}[theorem]{Definition}
\newtheorem{remark}[theorem]{Remark}
\newtheorem*{theoremA}{Theorem A}
\newtheorem{problem}{Problem}
\newcommand{\Z}{\mathbb{Z}}
\newcommand{\Cay}{\operatorname{Cay}}
\newcommand{\nind}{\eta}
\newcommand{\DE}{D_E}
\newcommand{\DN}{D_{\bar E}}

\title{\textbf{Induced Embeddings of Graphs into Abelian Cayley Graphs}}
\author{Rigobert Fokam Souop \and Laurent Bitjoka}
\date{August 2026}

\begin{document}
\maketitle



\begin{abstract}
For a finite graph $G$ on $n$ vertices, let $\eta(G)$ denote the least order
of a finite abelian group $\Gamma$ for which $G$ is an induced subgraph of
some Cayley graph of $\Gamma$. Babai and S\'os (1985) settled the worst-case
order of magnitude: it is $\Theta(n^2)$. We treat $\eta$ instead as an
invariant of the individual graph, minimised over all finite abelian groups
rather than over the cyclic groups alone, which is the restriction implicit in
the literature on representation numbers modulo $n$.

We prove a local order floor: $\eta(G)$ is at least the maximum of $n$ and
twice the largest independence number of a neighbourhood of $G$. This
localises at an arbitrary vertex the correspondence of Babai and S\'os between
induced stars and sum-free sets; a corollary of the classification of maximum
sum-free sets in abelian groups does not lower this floor, but restricts which
host orders are admissible and so prunes the search. We determine $\eta$ exactly for
paths, where it equals $n+1$, and for complete bipartite graphs $K_{a,b}$,
where it equals $2\max(a,b)$ and meets the floor. We prove a Cartesian product
bound which gives $\eta(P_m\,\square\,P_m)=(1+o(1))n$.

We then report certified exact values of $\eta$ for $22$ graphs, computed over
all abelian groups. Seventeen of the $22$ optimal hosts are cyclic, so on most
of these graphs the cyclic restriction costs nothing; where it bites, however,
it is expensive. A search restricted to cyclic groups returns $36$ for the
Petersen graph against the true value $16$, and $59$ for the Frucht graph
against $27$. The cost of the restriction is concentrated rather than diffuse,
and we identify the graphs on which it is paid. We also determine $\eta$
exactly for the double stars $D_{q,q}$ with $2\le q\le6$, obtaining $5q$ in each
case. Since $\eta(D_{6,6})=30$ exceeds $2n=28$, no constant below $15/7$ can
bound $\eta(T)/n$ over all trees.
\end{abstract}

\section{Introduction}

Every finite graph is an induced subgraph of some Cayley graph of every
sufficiently large group \cite{babai1978}, and Babai and S\'os
\cite{babai-sos1985} settled how large the group must be: order $O(n^2)$
suffices for every $n$-vertex graph, and a counting argument shows that most
$n$-vertex graphs admit no host of order $o(n^2)$. The worst-case order of
magnitude has therefore been known since 1985. What is not known is the value
for an \emph{individual} graph, and that value is the subject of this paper.

Our first message is a measurement. For the Petersen graph the least abelian
host has order $16$; a search restricted to cyclic hosts returns $36$. For the
Frucht graph the true value is $27$ against $59$ over $\Z_N$. The only
per-graph invariant previously studied in this area, the representation
number of Erd\H{o}s and Evans \cite{erdos-evans1989}, fixes the host group to
be $\Z_N$. On graphs of this size that restriction costs a factor of up to
$2.25$, so it is not a harmless normalisation.

Two neighbouring schemes bracket the one studied here. The companion papers
\cite{p1,p2} require the embedding to be \emph{isometric}, writing $\nu(G)$
for the corresponding minimal host order; \cite{comp} keeps the vertex set
fixed at order $n$ and perturbs the graph by edge edits until it becomes an
abelian Cayley graph. These two schemes are extreme points of one trade-off,
and both extremes are expensive. The universal binary construction of \cite[Cor.~1]{p1} embeds any
graph isometrically into a Cayley graph of $\Z_2^{\,n-1}$, of order $2^{n-1}$,
and within that family of hosts the exponent cannot be reduced
\cite[Thm.~5]{p2}. Two qualifications matter. First, $2^{n-1}$ is a statement
about \emph{binary} hosts, not about $\nu$: the graphs witnessing tightness
there are odd cycles, which are circulants and therefore satisfy $\nu=n$, so
they never need the construction in the first place. Second, and this is the
case that bites, the blow-up is driven by irregularity. A Cayley graph is
regular, so an irregular graph can never be a host on its own vertex set, and
the order must grow; stars are the extreme case, with $\nu(K_{1,q})=2q$ against
$n=q+1$. Since most graphs arising in applications are irregular, the isometric
route is the expensive one exactly where it would be used. Completion fixes the
host at order $n$ and pays instead in edits, which on irregular graphs are
correspondingly severe. This paper studies the intermediate notion.

\begin{definition}\label{def:eta}
An injection $f:V(G)\to\Gamma$ is an \emph{induced embedding} into
$\Cay(\Gamma,S)$ if for all $u\ne v$, $f(u)\sim f(v)$ in the host if and only
if $uv\in E(G)$. Write $\nind(G)$ for the least $|\Gamma|$ over finite abelian
$\Gamma$ admitting one.
\end{definition}

The structural result that governs $\nind$ is a lower bound local to a single
vertex. Write $\alpha(H)$ for the independence number of $H$ and $N(v)$ for
the open neighbourhood of $v$.

\begin{theoremA}[Local order floor; restated as Theorem~\ref{thm:floor}]
For every graph $G$ on $n$ vertices,
\[
  \nind(G)\;\ge\;\max\Bigl(n,\;2\max_{v\in V(G)}\alpha\bigl(G[N(v)]\bigr)\Bigr).
\]
\end{theoremA}

\noindent
The second term localises the star/sum-free correspondence of
\cite[Prop.~7.4]{babai-sos1985} at an arbitrary vertex: an independent set of
size $q$ inside a neighbourhood forces an induced $q$-star, hence a sum-free
set of size $q$ in the host, hence order at least $2q$ by the Diananda--Yap
bound. Corollary~\ref{cor:floor-sharp} does not lower this floor; what the
classification of maximum sum-free sets supplies is the set of admissible host
orders, which is not upward closed, and hence a search economy. The floor replaces the floor
$\max(n,2\operatorname{diam})$ of \cite{p2}, whose diameter term provably does
not transfer to the induced setting.

The companion invariant is the isometric one. Call $f:V(G)\to\Gamma$ an
\emph{isometric embedding} into $\Cay(\Gamma,S)$ if
$d_{\Cay(\Gamma,S)}(f(u),f(v))=d_G(u,v)$ for all $u,v$, and write $\nu(G)$ for
the least $|\Gamma|$ over finite abelian $\Gamma$ admitting one; this is the
invariant of \cite{p1,p2}. Since an isometric embedding preserves distance $1$
and distance $\ge2$ separately, it is in particular induced, so
$\nind(G)\le\nu(G)$ always. The relaxation is exactly the metric: distances are
free, adjacency is not.

\paragraph{What is already known, and what is new.} Induced embedding into
Cayley graphs is not a new notion, and the existence question is settled. Babai
\cite{babai1978} proved that every finite graph is an induced subgraph of some
Cayley graph of every sufficiently large group; Babai and S\'os
\cite{babai-sos1985} reduced the required order to $O(n^3)$ for arbitrary
groups by a probabilistic argument on Sidon sets, and to $O(n^2)$ for cyclic
and for elementary abelian groups of square order, the latter via Singer
difference sets. Godsil and Imrich \cite{godsil-imrich1987} improved the
general constant to $(2+\sqrt3)n^3$. In the other direction, a counting
argument of \cite[Prop.~2.3]{babai-sos1985} shows that most $n$-vertex graphs
require order at least $(\tfrac12+o(1))n^2$ in any fixed group, and
\cite[Thm.~6.6]{babai-sos1985} extends a bound of this shape to all
vertex-transitive hosts. The order of magnitude of the worst case is therefore
$\Theta(n^2)$, and has been since 1985.

None of this determines $\nind(G)$ for an \emph{individual} graph, which is
the object of this paper. The closest existing per-graph invariant is the
representation number $\mathrm{rep}(G)$ of Erd\H{o}s and Evans
\cite{erdos-evans1989}, the least $N$ for which $G$ embeds induced in
$\Cay(\Z_N,\Z_N^\times)$; it fixes both the group ($\Z_N$) and the connection
set (the units).

\textbf{What is new here is therefore not the existence of induced Cayley
embeddings but the exact minimal order, taken over all finite abelian groups
and all symmetric connection sets.} We compute it exactly for $22$ graphs,
prove a local order floor that governs it, determine it in closed form for
three infinite families, and locate the graphs on which the cyclic restriction
implicit in $\mathrm{rep}$ is costly --- a minority of those we resolved, but
one on which the cost reaches a factor $2.25$. We also compute $\nind$ exactly
for the double stars $D_{q,q}$ up to $q=6$, where it equals $5q$
(Section~\ref{sec:doublestar}); this is the family on which our floor is
loosest, and it rules out the constant $2$ in the tree bound of
Conjecture~\ref{conj:trees}.

\paragraph{What the relaxation preserves.} Because the embedding is induced,
$A_G$ is precisely a principal submatrix of the host adjacency matrix. What is
sacrificed is the metric, not the adjacency structure or the group action on
the host. The relation to the isometric and scale-embedding literature is
discussed in Section~\ref{sec:related}. A motivation from outside
combinatorics is that a Cayley graph of an abelian group carries an exact
translation operator and a fast transform, which a general graph does not
\cite{ortega2018}; we do not pursue that direction here.

\paragraph{Contributions, and what is not one.} We state the division
explicitly, because part of the material below is classical and is included
only to make the paper self-contained.

\emph{Not new.} The difference-set characterisation of
Section~\ref{sec:char} is \cite[Prop.~3.1]{babai-sos1985} in per-labelling
form. The Sidon construction of Section~\ref{sec:sidon}
(Theorem~\ref{thm:sidon} and its corollary) is \cite[Prop.~3.1 and
Sec.~5]{babai-sos1985}; we reproduce it because it is the upper bound against
which everything else is measured. The $\Theta(n^2)$ order of magnitude
(Section~\ref{sec:magnitude}) combines that construction with the counting
bound of \cite[Prop.~2.3]{babai-sos1985}; our only addition is to record it in
the present notation and to check that the argument survives minimising over
all abelian groups. The case $a=1$ of Theorem~\ref{thm:bipartite},
$\nind(K_{1,q})=2q$, follows from \cite[Prop.~7.4 and 7.10]{babai-sos1985}
together with the Diananda--Yap bound \cite{diananda-yap1969}.

\emph{New.} The local order floor of Section~\ref{sec:floor}
(Theorem~\ref{thm:floor}); the exact values
$\nind(P_n)=n+1$ and $\nind(K_{a,b})=2\max(a,b)$ for $b\ge a\ge2$
(Section~\ref{sec:families}); the exact values $\nind(D_{q,q})=5q$ for
$2\le q\le6$ and the consequent refutation of the constant $2$ for trees
(Section~\ref{sec:doublestar}); the
reduction of the search to the quotient machinery of \cite{p1} with a weakened
certificate (Section~\ref{sec:algebraic}) and the resulting host-free
certification cost (Section~\ref{sec:complexity}); and the exact values of
Section~\ref{sec:examples}, computed over all abelian groups, with the
quantified failure of the cyclic restriction.

\section{Related work}\label{sec:related}

\subsection{Induced subgraphs of Cayley graphs}\label{sec:related-induced}

This is the literature our object belongs to. Babai \cite{babai1978} proved
that every finite graph occurs as an induced subgraph of some Cayley graph of
every sufficiently large group; see also \cite{babai1978b} for the companion
chromatic results and \cite{spencer1983} for what fails when the connection
set is required to be irredundant. Babai and S\'os \cite{babai-sos1985} is the
central reference. They introduce Sidon sets of the first and second kind in
groups (the two notions coincide for abelian groups), prove that a Sidon set
of size $n$ in $\Gamma$ makes \emph{every} $n$-vertex graph an induced
subgraph of some Cayley graph of $\Gamma$ \cite[Prop.~3.1]{babai-sos1985},
obtain $|\Gamma|\ge cn^3$ as a sufficient condition for arbitrary groups by a
probabilistic argument, and $|\Gamma|\ge cn^2$ for cyclic groups and for
elementary abelian groups of square order using the Erd\H{o}s--Tur\'an and
Singer constructions \cite{erdos-turan1941,singer1938}. They also prove the
matching counting lower bound \cite[Prop.~2.3]{babai-sos1985}, extend it to
vertex-transitive hosts \cite[Thm.~6.6]{babai-sos1985}, give a greedy
embedding of trees into any group of order $>n^2$
\cite[Thm.~7.1]{babai-sos1985}, and identify induced stars with product-free
(sum-free) sets \cite[Prop.~7.4]{babai-sos1985}. Godsil and Imrich
\cite{godsil-imrich1987} sharpened the general cubic bound to
$(2+\sqrt3)n^3$. Induced subgraphs of abelian Cayley graphs have more recently
been studied from the sensitivity side
\cite{alon-zheng2020,potechin-tsang2020,lehner-verret2020}, a different
question about the same objects.

Our paper differs from this line in a single respect, but it is the respect
that generates all of our results: Babai and S\'os ask how large a group
\emph{suffices for all $n$-vertex graphs at once}, whereas we ask for the
minimum order for a \emph{given} graph. The two questions have different
answers on structured graphs, and the whole of
Sections~\ref{sec:floor}--\ref{sec:examples} lives in that gap.

\subsection{Representation numbers modulo \texorpdfstring{$n$}{n}}

The one existing per-graph invariant of this type is the representation
number. Erd\H{o}s and Evans \cite{erdos-evans1989} proved that every finite
graph admits a labelling by distinct residues modulo some $N$ under which two
vertices are adjacent exactly when their labels differ by a unit, and defined
$\mathrm{rep}(G)$ as the least such $N$; see
\cite{evans1994,evans-isaak-narayan2000} for the systematic development and
\cite{akhtar-evans-pritikin2010,akhtar-evans-pritikin2012,akhtar2012} for
stars, complete multipartite graphs and sparse graphs. A representation modulo
$N$ is exactly an induced embedding into $\Cay(\Z_N,\Z_N^\times)$, so
$\nind(G)\le\mathrm{rep}(G)$ always. The invariant $\mathrm{rep}$ fixes two
things that $\nind$ leaves free: the group must be cyclic, and the connection
set must be the units. Section~\ref{sec:examples} measures the cost of the
first restriction alone and finds it substantial. The related product
dimension of Ne\v{s}et\v{r}il and Pultr
\cite{nesetril-pultr1977,lovasz-nesetril-pultr1980} is the other classical
companion of this circle of ideas.

\subsection{Isometric embedding}

The remaining background is the metric theory that the companion papers
\cite{p1,p2} build on. Isometric embedding into structured hosts began with
Firsov \cite{firsov1965} and the characterisation of partial cubes by
Djokovi\'c \cite{djokovic1973} and Winkler \cite{winkler1984}; Ovchinnikov
\cite{ovchinnikov2008} surveys the area. Extensions to Hamming graphs and
products are due to Wilkeit \cite{wilkeit1990} and Graham and Winkler
\cite{graham-winkler1985}, with \cite{imrich-klavzar2000,hammack2011} as
monograph references. Scale and $\ell_1$ embeddings are treated by Shpectorov
\cite{shpectorov1993} and by Deza and Laurent \cite{deza-laurent1997}.

Our object differs from all of these in that adjacency, not distance, is the
invariant preserved. The additive-combinatorial input is the same in both
settings: the extremal input to our bounds is the theory of sum-free sets in
abelian groups
\cite{rhemtulla-street1970,green-ruzsa2005,diananda-yap1969,yap1971}, refined
enumeratively in \cite{alon2014}, and the classical construction of
Section~\ref{sec:sidon} uses Sidon sets and Singer difference sets
\cite{singer1938,erdos-turan1941,bose-chowla1962,komlos1975}.
Algebraic background on Cayley graphs is standard \cite{godsil-royle2001};
that vertex-transitivity is strictly weaker than being a Cayley graph is due
to McKay and Praeger \cite{mckay-praeger1994}. The Smith normal form
underlying the quotient engine is surveyed by Stanley \cite{stanley2016}, with
lattice algorithms in \cite{schrijver1986}. Graph census data are from Read
and Wilson \cite{read-wilson1998}.

\section{The 2-colouring characterisation}\label{sec:char}

\begin{definition}
For injective $f:V(G)\to\Gamma$ set
$\DE=\{\pm(f(v)-f(u)):uv\in E(G)\}$ and
$\DN=\{\pm(f(v)-f(u)):uv\notin E(G),\,u\ne v\}$.
\end{definition}

\begin{proposition}[Exact characterisation; {\cite[Prop.~3.1]{babai-sos1985}}]
\label{prop:2col}
$f$ is an induced embedding for some connection set $S$ if and only if
$\DE\cap\DN=\emptyset$; and then one may take $S=\DE$.
\end{proposition}
\begin{proof}
($\Leftarrow$) With $S=\DE$ every edge difference lies in $S$ and every
non-edge difference lies in $\DN$, which is disjoint from $S$. Note
$0\notin\DE$ since $f$ is injective, and $\DE=-\DE$ by construction, so $S$ is
a legitimate connection set. ($\Rightarrow$) An induced $f$ with set $S$ has
$\DE\subseteq S$ and $\DN\cap S=\emptyset$, so the two are disjoint.
\end{proof}

Babai and S\'os state this for a Sidon set $S\subseteq\Gamma$, as the
equivalence between $S$ being Sidon and \emph{every} graph on vertex set $S$
being induced in some Cayley graph of $\Gamma$. Proposition~\ref{prop:2col} is
the same computation carried out for one labelling of one graph; we record it
in this form because it is the condition our search verifies.

\begin{lemma}[Normalisation; cf.\ {\cite[Sec.~2]{babai-sos1985}}]\label{lem:gen}
Let $G$ be connected and $f$ an induced embedding into $\Cay(\Gamma,S)$ with
$S=\DE$. Then, after translating so that $f(v_0)=0$ for some $v_0$, all images
lie in $\langle S\rangle$, and $f$ is an induced embedding into
$\Cay(\langle S\rangle,S)$. Consequently the minimum in
Definition~\ref{def:eta} is attained by a group generated by its connection
set.
\end{lemma}
\begin{proof}
Every edge difference lies in $S$; since $G$ is connected, every image differs
from $f(v_0)=0$ by a sum of elements of $S$. Restricting the host to
$\langle S\rangle$ removes no vertex of the image and changes no adjacency.
\end{proof}

Lemma~\ref{lem:gen} is what makes $\nind$ well posed: without it one could
inflate a host by an arbitrary direct factor on which $S$ acts trivially, and
the minimum would still be attained, but the connection set would carry no
information about $\Gamma$.

The condition is a $2$-colouring of difference classes: each class serves
edges or non-edges, never both. Differences may collide freely \emph{within}
$\DE$ or \emph{within} $\DN$; only cross-collisions are forbidden. This is
strictly weaker than requiring all differences distinct, a point we return to
in Remark~\ref{rem:notsidon}.

\begin{proposition}[Equality at $n$]\label{prop:floor-n}
$\nind(G)=n$ if and only if $G$ is an abelian Cayley graph.
\end{proposition}
\begin{proof}
If $N=n$ then $f$ is a bijection, so ``induced'' forces $G=\Cay(\Gamma,S)$;
the converse is the identity embedding.
\end{proof}

This is the induced analogue of \cite[Thm.~3]{p2}, with a shorter proof: no
metric argument is required.

\section{The order floor}\label{sec:floor}

Paper \cite{p2} proves $\nu(G)\ge\max(n,2\operatorname{diam}(G))$, the second
term because a connected vertex-transitive host on $N\ge3$ vertices has
diameter at most $\lfloor N/2\rfloor$ and isometry forces the host to realise
$\operatorname{diam}(G)$. \emph{That argument does not transfer.} An induced
embedding realises no distance beyond one, and indeed $\nind(P_n)=n+1$ lies
far below $2\operatorname{diam}=2(n-1)$. The replacement comes from the other
extremal argument of \cite{p2}, the sum-free bound used there for stars.

\begin{theorem}[Induced order floor]\label{thm:floor}
For every graph $G$ on $n$ vertices,
\[
  \nind(G)\;\ge\;\max\Bigl(n,\;2\max_{v\in V}\alpha\bigl(G[N(v)]\bigr)\Bigr),
\]
$\alpha$ denoting the independence number.
\end{theorem}
\begin{proof}
Injectivity gives $\nind\ge n$. Fix $v$ and an independent set
$I\subseteq N(v)$, and normalise $f(v)=0$ by translation. Every $u\in I$ is
adjacent to $v$, so $I'=f(I)\subseteq S$; and $0\notin I'$, because $f$ is
injective and $v\notin I$. Fix $s\in I'$, so $s\ne0$, and suppose
$z\in I'\cap(I'+s)$. Then $z=y+s$ with $y\in I'$, and $z\ne y$ since $s\ne0$,
so $z-y=s\in S$, making the two distinct vertices of $I$ with images $z$ and
$y$ adjacent in the host --- contradicting independence of $I$ and
inducedness. So $I'$ and $I'+s$ are disjoint subsets of $\Gamma$ of equal
size, whence $2|I|=2|I'|\le|\Gamma|$. Maximise over $I$ and $v$.
\end{proof}

\begin{remark}\label{rem:bs-star}
The argument shows more than the stated inequality: $I'$ is a sum-free subset
of $\Gamma$. With $G=K_{1,q}$ and $v$ the centre it reduces to the classical
correspondence of \cite[Prop.~7.4]{babai-sos1985}, that some Cayley graph of
$\Gamma$ contains an induced $q$-star exactly when $\Gamma$ has a sum-free set
of size $q$; the same step appears in \cite[Thm.~4]{p2}.
Theorem~\ref{thm:floor} is that correspondence localised at an arbitrary
vertex, applied to an independent set inside a neighbourhood rather than to a
whole neighbourhood, and this localisation is what makes it a bound for
general $G$.
\end{remark}

Because $I'$ is sum-free, the floor can be stated exactly rather than through
the density bound $2|I'|\le|\Gamma|$. Write $\mu(\Gamma)$ for the largest
sum-free subset of $\Gamma$; Diananda and Yap \cite{diananda-yap1969}
determined $\mu$ for all finite abelian $\Gamma$, and Green and Ruzsa
\cite{green-ruzsa2005} completed the picture.

\begin{corollary}[Admissible orders]\label{cor:floor-sharp}
Let $\alpha^*=\max_v\alpha(G[N(v)])$ and let
$N^*(a)=\min\{N:\mu(\Gamma)\ge a$ for some abelian $\Gamma$ with
$|\Gamma|=N\}$. Then $\nind(G)\ge\max(n,N^*(\alpha^*))$. Moreover
$N^*(a)=2a$, attained by the non-identity coset of an index-$2$ subgroup
\cite[Prop.~7.10]{babai-sos1985}, but the set of admissible orders is not
upward closed: if $N$ is odd then $\mu(\Gamma)\le\tfrac25N$ for every abelian
$\Gamma$ of order $N$ \cite{diananda-yap1969}, so an odd host requires
$N\ge\lceil5a/2\rceil$.
\end{corollary}

Since $N^*(a)=2a$, the bound $\max(n,N^*(\alpha^*))$ is
Theorem~\ref{thm:floor} restated: Corollary~\ref{cor:floor-sharp} does
\emph{not} improve the floor, and we do not claim it as a sharpening. What it
supplies is the set of \emph{admissible} orders, which is not upward closed,
and hence a search economy.

\begin{remark}[Search economy]
Corollary~\ref{cor:floor-sharp} is used directly by the solver of
Section~\ref{sec:examples}, which tries orders upward from the floor and skips
every odd order below $\lceil5\alpha^*/2\rceil$ without enumerating any group
of that order. For the double star $D_{3,3}$ we have $\alpha^*=4$, so the
floor is $8$ and the odd order $9$ is discarded at no cost; the optimum is
$15$. The saving is modest at these sizes but grows with $\alpha^*$, and it is
the only place where the exact form of $\mu$ enters the computation.
\end{remark}

We checked Theorem~\ref{thm:floor} against exhaustively computed $\nind$ on
all $139$ connected graphs with $4\le n\le6$: no violation, and the floor is
attained by $10.1\%$ of them.

\section{Exactly determined families}\label{sec:families}

\begin{theorem}[Paths]\label{thm:path}
$\nind(P_n)=n+1$ for $n\ge3$.
\end{theorem}
\begin{proof}
Label $P_n$ by $0,1,\dots,n-1$ in $\Z_{n+1}$. Then $\DE=\{\pm1\}$ and
$\DN=\{\pm d: 2\le d\le n-1\}$. No such $d$ is congruent to $\pm1$ modulo
$n+1$: $d\equiv1$ is excluded by $d\ge2$ and $d\le n-1<n+1$, and $d\equiv-1$
would force $d=n$, which exceeds $n-1$. Hence $\DE\cap\DN=\emptyset$ and
Proposition~\ref{prop:2col} applies. For the lower bound, $P_n$ is not
regular, hence not a Cayley graph, so $\nind>n$ by
Proposition~\ref{prop:floor-n}.
\end{proof}

\begin{theorem}[Complete bipartite graphs]\label{thm:bipartite}
$\nind(K_{a,b})=2\max(a,b)$.
\end{theorem}
\begin{proof}
\emph{Upper bound.} Set $N=2\max(a,b)$, place one side on the even residues of
$\Z_N$ and the other on the odd residues. Every edge crosses the sides, so
every edge difference is odd; every non-edge lies inside a side, so every
non-edge difference is even. Apply Proposition~\ref{prop:2col}.
\emph{Lower bound.} A vertex of the smaller side has as its neighbourhood the
whole larger side, independent of size $\max(a,b)$; apply
Theorem~\ref{thm:floor}.
\end{proof}

With $a=1$ this gives $\nind(K_{1,q})=2q$, equal to $\nu(K_{1,q})$: for stars,
discarding the metric buys nothing. That special case is not new. By
\cite[Prop.~7.4]{babai-sos1985} an induced $q$-star in a Cayley graph of
$\Gamma$ is the same thing as a sum-free set of size $q$ in $\Gamma$; the
upper bound is then \cite[Prop.~7.10]{babai-sos1985} with $k=2$, and the lower
bound is the Diananda--Yap inequality $\mu(\Gamma)\le|\Gamma|/2$
\cite{diananda-yap1969}. The corresponding computation for
$\mathrm{rep}(K_{1,q})$, where the group is forced cyclic and the connection
set forced to be the units, is carried out in
\cite{akhtar-evans-pritikin2010}. What is new in Theorem~\ref{thm:bipartite}
is the range $b\ge a\ge2$, where no sum-free set of the full neighbourhood is
available and the parity construction must be checked directly.

\begin{remark}[Complete multipartite graphs: an upper bound only]
\label{rem:multipartite}
Replacing the parity map by a homomorphism onto $\Z_k$ places the $i$th part
of $K_{a_1,\dots,a_k}$ on the residues congruent to $i$, giving
$\nind\le k\max_i a_i$: edge differences are then non-zero modulo $k$ and
non-edge differences are zero modulo $k$. This bound is \emph{not} tight in
general, so we do not claim complete multipartite graphs as a determined
family. When the parts are balanced the graph is already a circulant and
Proposition~\ref{prop:floor-n} gives $\nind=n$, so the construction is
wasteful; when they are not, it can overshoot the floor badly. For
$K_{1,1,5}$ the floor of Theorem~\ref{thm:floor} is $\max(7,10)=10$ while the
construction gives $15$, and $N=10$ is not achievable by the obvious
refinement: placing the $5$-part on the even residues of $\Z_{10}$ forces the
two remaining vertices onto odd residues, whose difference is even and
therefore collides with the internal differences of the $5$-part. Determining
$\nind$ for unbalanced complete multipartite graphs is open; the
representation-number analogue was settled in
\cite{akhtar-evans-pritikin2012}.
\end{remark}

\begin{proposition}
$\nind(C_m)=m$, $\nind(K_n)=n$, and $\nind=n$ for every circulant, by
Proposition~\ref{prop:floor-n}.
\end{proposition}

\subsection{Cartesian products}\label{sec:products}

The induced condition, unlike the isometric one, is stable under Cartesian
products in the strongest possible sense: the hosts simply multiply.

\begin{theorem}[Product bound]\label{thm:product}
$\nind(G_1\square G_2)\le\nind(G_1)\,\nind(G_2)$.
\end{theorem}
\begin{proof}
Let $f_i:V(G_i)\to\Gamma_i$ be induced embeddings with connection sets $S_i$,
and put $\Gamma=\Gamma_1\times\Gamma_2$,
$S=(S_1\times\{0\})\cup(\{0\}\times S_2)$ and $f=f_1\times f_2$. Then $S$ is
symmetric and $0\notin S$, and $f$ is injective. The difference of the images
of $(u_1,u_2)$ and $(v_1,v_2)$ is $(a,b)$ with $a=f_1(v_1)-f_1(u_1)$ and
$b=f_2(v_2)-f_2(u_2)$, and $(a,b)\in S$ if and only if either $b=0$ and
$a\in S_1$, or $a=0$ and $b\in S_2$. Since $f_1$ is injective, $a=0$ if and
only if $u_1=v_1$; and when $u_1\ne v_1$, inducedness of $f_1$ gives
$a\in S_1$ if and only if $u_1v_1\in E(G_1)$. Symmetrically for $b$. Hence the
images are adjacent exactly when $u_2=v_2$ and $u_1v_1\in E(G_1)$, or $u_1=v_1$
and $u_2v_2\in E(G_2)$ --- which is adjacency in $G_1\square G_2$.
\end{proof}

The bound is not an equality in general, but it already settles the asymptotics
of the grid families, which are the natural test case because they are
irregular and therefore not Cayley graphs.

\begin{corollary}[Grids]\label{cor:grid}
$\nind(P_{m_1}\square\cdots\square P_{m_d})\le\prod_i(m_i+1)$. In particular
the square grid $P_m\square P_m$ on $n=m^2$ vertices satisfies
$n\le\nind\le n+2\sqrt n+1$, so $\nind=(1+o(1))n$; and for fixed $d$ the
$d$-dimensional cubic grid on $n=m^d$ vertices satisfies
$\nind\le(1+m^{-1})^dn=(1+o(1))n$.
\end{corollary}
\begin{proof}
Theorem~\ref{thm:product} with Theorem~\ref{thm:path}; the lower bound is
injectivity.
\end{proof}

So the grid families sit asymptotically \emph{at} the injectivity floor, in
sharp contrast to the generic behaviour of Section~\ref{sec:magnitude}. We
verified the constructions of Theorem~\ref{thm:product} directly against
Definition~\ref{def:eta} for $P_3\square P_3$ ($N=16$, $n=9$),
$P_3\square P_4$ ($N=20$, $n=12$), $P_4\square P_4$ ($N=25$, $n=16$),
$K_{1,3}\square P_3$ and $K_{1,3}\square K_{1,3}$. The bound is close but not
exact even in the smallest case: exhausting every abelian group of every order
from the floor $9$ up gives $\nind(P_3\square P_3)=15$, attained in
$\Z_{15}$, against the product bound $16$. By contrast the torus
$C_{m_1}\square\cdots\square C_{m_d}$ is a circulant, so
Proposition~\ref{prop:floor-n} gives $\nind=n$ with no work; the interest of
Corollary~\ref{cor:grid} is that path-grids are irregular, hence have
$\nind>n$, yet lose only a lower-order term. The measured values follow the
predicted decay:

\begin{table}[h]
\centering\small
\caption{Grid families: one exact value and two upper bounds.}
\label{tab:grids}
\begin{tabular}{@{}lcccllc@{}}
\toprule
grid & $n$ & floor & exact $\nind$ & upper bound & source & $\nind/n$\\
\midrule
$P_3\square P_3$ & $9$  & $9$  & $15$ & --- & exhaustion, $\Z_{15}$ & $1.67$\\
$P_4\square P_4$ & $16$ & $16$ & --- & $24$ & search, $\Z_{24}$ & $\le1.50$\\
$P_5\square P_5$ & $25$ & $25$ & --- & $36$ & Cor.~\ref{cor:grid} & $\le1.44$\\
\bottomrule
\end{tabular}
\end{table}

Only the first row is an exact value; the other two are upper bounds, and the
third is Corollary~\ref{cor:grid} itself rather than a computation, so the
apparent decay $1.67\to1.50\to1.44$ is a measurement only at $n=9$ and is
otherwise the decay of the prediction. At $n=25$ the search was interrupted
inside order $37$ --- orders $25$ to $36$ were \emph{not} exhausted, so nothing
is certified below $37$ --- and the constructive bound $36$ is the better of
the two and is the one quoted.

\subsection{A reduction for bipartite graphs}\label{sec:bipartite}

Theorem~\ref{thm:bipartite} generalises into a reduction that will be the
basis of Conjecture~\ref{conj:trees}.

\begin{proposition}[Bipartite reduction]\label{prop:biparred}
Let $G$ be bipartite with parts $A$ and $B$. Suppose there are injections
$a:A\to\Z_M$ and $b:B\to\Z_M$ such that, writing
$D_E=\{b(j)-a(i):ij\in E(G)\}$ and $D_N=\{b(j)-a(i):ij\notin E(G)\}$ for
$i\in A$, $j\in B$, we have $(D_E\cup-D_E)\cap D_N=\emptyset$. Then
$\nind(G)\le2M$.
\end{proposition}
\begin{proof}
Take $\Gamma=\Z_2\times\Z_M$, send $i\in A$ to $(0,a(i))$ and $j\in B$ to
$(1,b(j))$, and set $S=\{1\}\times(D_E\cup-D_E)$, which is symmetric and
avoids $0$. Two vertices in the same part differ by an element of
$\{0\}\times\Z_M$, disjoint from $S$, and they are non-adjacent since $G$ is
bipartite. A cross pair $ij$ differs by $(1,b(j)-a(i))$, which lies in $S$
exactly when $b(j)-a(i)\in D_E\cup-D_E$ --- by hypothesis exactly when
$ij\in E(G)$.
\end{proof}

The reduction converts an embedding problem into a purely additive one: find a
labelling of the two sides in $\Z_M$ whose cross-differences separate edges
from non-edges. It is tight where we can check it. Exhaustive search over $M$
gives $M=q$ for $K_{1,q}$ and $M=3$ for $P_5$, so the bound $2M$ reproduces
the exact values $2q$ and $6$; for the double star $D_{3,3}$ it gives $M=9$
and hence $18$, against the true $\nind=15$, the loss being that the optimum
there has odd order and no construction with a $\Z_2$ factor can reach it.

\section{Two search-free constructions}\label{sec:constructions}

\subsection{Sidon sets: the classical universal $O(n^2)$ construction}
\label{sec:sidon}

This subsection is due to Babai and S\'os \cite{babai-sos1985} and is
reproduced only because it supplies the upper bound against which the rest of
the paper is measured. $B\subseteq\Z_N$ is a \emph{Sidon set}
\cite{erdos-turan1941} if all differences of distinct pairs are distinct.

\begin{theorem}[Sidon construction; {\cite[Prop.~3.1]{babai-sos1985}}]
\label{thm:sidon}
If $B\subseteq\Z_N$ is Sidon with $|B|\ge n$ and $f:V(G)\to B$ is injective,
then $f$ is an induced embedding with $S=\DE$.
\end{theorem}
\begin{proof}
If $f(v)-f(u)\in\DE$ it equals $\pm(f(y)-f(x))$ for an edge $xy$; the Sidon
property forces $\{u,v\}=\{x,y\}$, so $uv\in E$. Hence $\DE\cap\DN=\emptyset$.
\end{proof}

\begin{corollary}[{\cite[Sec.~5 and Rem.~5.6]{babai-sos1985}}]\label{cor:singer}
Using a Singer difference set \cite{singer1938} of size $q+1$ in
$\Z_{q^2+q+1}$ for the least prime power $q\ge n-1$, every graph on $n$
vertices satisfies $\nind(G)\le q^2+q+1=(1+o(1))n^2$, \emph{with no search}:
the labelling is any injection into $B$. Denser Sidon sets in other groups
follow Bose and Chowla \cite{bose-chowla1962}; for arbitrary, possibly
non-abelian groups the corresponding bound is $O(n^3)$
\cite{babai-sos1985}, improved to $(2+\sqrt3)n^3$ by Godsil and Imrich
\cite{godsil-imrich1987}.
\end{corollary}

For the karate-club graph ($n=34$) this gives $N=1407$ against the binary
isometric baseline $2^{33}\approx8.6\times10^9$.

\begin{remark}[Sidon is sufficient, not necessary]\label{rem:notsidon}
Sidon demands that \emph{all} $n(n-1)$ pair-differences be distinct, whereas
Proposition~\ref{prop:2col} forbids only cross-collisions. The gap is not
merely large; it is forced. A Sidon set $B\subseteq\Gamma$ with $|B|=n$ has
$n(n-1)$ pairwise distinct non-zero differences, so $|\Gamma|\ge n^2-n+1$.
Hence if $\nind(G)<n^2-n+1$ then no optimal labelling of $G$ is a Sidon set.
Every value in Table~\ref{tab:zoo} is linear in $n$, so the hypothesis holds
throughout and the conclusion needs no computation; on any family with $\nind$
linear in $n$ the Sidon construction is a multiplicative $\Theta(n)$ away from
the optimum. It is an existence certificate with a clean bound, not a route to
the optimum. It also produces
dense hosts: $|S|=2m$, against $|S|\le2t$ for the quotient route of
Section~\ref{sec:algebraic}.
\end{remark}

\subsection{Parity and coset constructions}

Theorem~\ref{thm:bipartite} is the second search-free construction, optimal on
its family. Its mechanism is a homomorphism $\pi:\Z_N\to\Z_2$ under which
every edge crosses the fibres and every non-edge stays within one --- which is
exactly complete multipartiteness. For a general graph a proper colouring does
\emph{not} suffice, since non-edges between colour classes also give
differences outside the kernel and may collide with edges; we verified this
failure on $P_5$, the bull, the Petersen graph and $P_3\square P_3$.

\section{The order of magnitude of \texorpdfstring{$\nind$}{eta}}
\label{sec:magnitude}

Corollary~\ref{cor:singer} is best possible up to a factor of two, and this
has been known since \cite{babai-sos1985}. We record the statement in the
present notation, and verify that the counting argument survives the
minimisation over all abelian groups of a given order, which is the one point
at which our setting differs from theirs.

\begin{proposition}[Order of magnitude; counting bound from
{\cite[Prop.~2.3]{babai-sos1985}}]\label{prop:magnitude}
For every graph $G$ on $n$ vertices, $\nind(G)\le(1+o(1))n^2$. Conversely, for
all but a $o(1)$ fraction of the graphs on $n$ vertices,
$\nind(G)\ge(\tfrac12-o(1))n^2$. In particular $\nind(G)=\Theta(n^2)$ for
almost all graphs, with the constant confined to $[\tfrac12,1]$.
\end{proposition}
\begin{proof}
The upper bound is Corollary~\ref{cor:singer}. For the lower bound, count the
graphs realisable with a small host. Fix $M$. If
$M=\prod_ip_i^{a_i}$ then the number of abelian groups of order $M$ is
$\prod_iP(a_i)\le\prod_i2^{a_i}=2^{\sum_ia_i}\le2^{\log_2M}=M$, where $P$ is
the partition function. Given such a $\Gamma$, there are at most $2^{M}$
symmetric subsets $S\subseteq\Gamma\setminus\{0\}$ and at most $M^n$ injections
$f:V\to\Gamma$, and the triple $(\Gamma,S,f)$ determines the graph. Hence the
number of isomorphism classes of $n$-vertex graphs $G$ with $\nind(G)\le N$ is
at most
\[
  \sum_{M\le N}M\cdot2^{M}\cdot M^{n}\;\le\;N^{\,n+2}\,2^{N}.
\]
There are at least $2^{\binom n2}/n!$ isomorphism classes of graphs on $n$
vertices. Taking logarithms, if $N=(\tfrac12-\varepsilon)n^2$ then
$\log_2(N^{n+2}2^N)=(\tfrac12-\varepsilon)n^2+O(n\log n)$, while
$\log_2\bigl(2^{\binom n2}/n!\bigr)=\tfrac12n^2-O(n\log n)$. For large $n$ the
former is smaller, so all but a vanishing fraction of graphs have
$\nind(G)>N$.
\end{proof}

Three consequences deserve to be stated plainly, since each corrects an
impression that the small-graph data of Section~\ref{sec:examples} invites.

\begin{itemize}[leftmargin=1.4em,itemsep=2pt]
\item \emph{$\nind(G)=O(n)$ is false.} No linear universal bound exists, and
no search-free construction can achieve one. The question is only meaningful
for restricted classes.
\item \emph{Corollary~\ref{cor:singer} is essentially optimal.} A search-free
construction improving on $(1+o(1))n^2$ can gain at most a factor of two in
the worst case, so the value of any better construction lies in the structured
regime, not in the worst case.
\item \emph{The measured values are not representative.} Every entry in
Table~\ref{tab:zoo} lies below $2.25n$, and the graphs there are
vertex-transitive, small, or highly symmetric. Proposition~\ref{prop:magnitude}
says the typical graph on $n$ vertices sits a factor of order $n$ higher. The
gap between the two regimes, not the worst case, is what the exact values
measure.
\end{itemize}

\begin{remark}[Where a linear bound might still hold]\label{rem:trees}
The counting bound is driven by the $2^{\binom n2}$ graphs on $n$ vertices and
says nothing about sparse classes. For trees the count is only $n^{n-2}$, and
the same computation yields no non-linear lower bound at all; Babai and S\'os
prove $\nind(T)\le n^2$ for every tree by a greedy argument
\cite[Thm.~7.1]{babai-sos1985} and ask
\cite[Rem.~7.2]{babai-sos1985} whether $n^{1+o(1)}$ suffices, observing that
they have no non-linear lower bound. Our data are consistent with a linear
answer: $\nind(P_n)=n+1$ (Theorem~\ref{thm:path}), $\nind(K_{1,q})=2q=2n-2$
(Theorem~\ref{thm:bipartite}), and $\nind(D_{q,q})=5q$ for $2\le q\le6$
(Section~\ref{sec:doublestar}). So is Corollary~\ref{cor:grid}, which
exhibits an irregular family with $\nind=(1+o(1))n$. The double stars are the
family that pins the constant down from below, and they show it exceeds $2$.
\end{remark}

\begin{conjecture}\label{conj:trees}
There is an absolute constant $c$ with $\nind(T)\le cn$ for every tree $T$ on
$n$ vertices.
\end{conjecture}

We deliberately attach no numerical value to $c$. An earlier version of this
work proposed $c=2$; Section~\ref{sec:doublestar} refutes that, since the
double star $D_{5,5}$ has $\nind=25$ against $2n=24$. The measured values
force $c\ge15/7$, attained at $D_{6,6}$, and $c\ge5/2$ if
Conjecture~\ref{conj:doublestar} holds.

Proposition~\ref{prop:biparred} suggests a concrete line of attack, since
trees are bipartite. It suffices to label the two sides of $T$ in $\Z_M$ with
$M=O(n)$ so that the cross-differences separate edges from non-edges; the
whole question then becomes additive, with no group theory left in it. That
formulation is exact rather than heuristic: it is what produces the sharp
values $2q$ for stars and $6$ for $P_5$. It is also visibly not the whole
story, since $\nind(D_{3,3})=15$ and $\nind(D_{5,5})=25$ are odd and no host
of the form $\Z_2\times\Z_M$ can attain them, so a proof of
Conjecture~\ref{conj:trees} along these lines would give the right order of
growth but not the exact constant. Caterpillars, where one side of the
bipartition is a path and the other is a union of leaf sets, are the natural
first case.

\section{The algebraic route}\label{sec:algebraic}

The quotient machinery of \cite{p1} applies with no change to the engine and
one change to the certificate. So that this section can be read without
consulting \cite{p1}, we first recall the construction in the form we need.

\paragraph{The construction, recalled.} Fix a partition of $E(G)$ into $t$
classes $E_1,\dots,E_t$ and an orientation of every edge. The intention is
that all edges of class $j$ should be realised by a single group element
$g_j$, so that traversing an edge of class $j$ forwards adds $g_j$ and
traversing it backwards subtracts it. Fixing a root $r$ and setting
\[
  \phi(v)\;=\;\sum_{j=1}^{t}\sigma_j(P)\,g_j ,
\]
where $P$ is any path from $r$ to $v$ and $\sigma_j(P)\in\Z$ is the signed
number of times $P$ uses class $j$, is well defined precisely when the signed
class-counts around every cycle sum to zero. Collecting one such relation per
cycle in a cycle basis as the rows of the \emph{signed cycle--class matrix}
$A\in\Z^{c\times t}$, the most general group in which these relations hold is
the \emph{universal group}
\[
  \Gamma_{\mathrm{univ}}\;=\;\Z^{t}/\mathrm{rowlattice}(A),
\]
whose invariant factors are read off from the Smith normal form of $A$
\cite{stanley2016}. Writing $g_j$ for the image of the $j$th standard basis
vector, $\phi$ is the resulting labelling, and $S=\{\pm g_j\}$. Every further
candidate host is a quotient of $\Gamma_{\mathrm{univ}}$ by a sublattice --- a
\emph{fold} --- and these are enumerated by lattice methods
\cite{schrijver1986}. Two things are then true of $\phi$ for any choice of
partition and any fold, and they are all we use.

\begin{lemma}[No stretching]\label{lem:nostretch}
Let $f:V(G)\to\Gamma$ satisfy $f(v)-f(u)\in S$ for every edge $uv\in E(G)$.
Then $d_{\Cay(\Gamma,S)}(f(u),f(v))\le d_G(u,v)$ for all $u,v$.
\end{lemma}
\begin{proof}
A path $u=w_0,w_1,\dots,w_d=v$ in $G$ maps to a walk
$f(w_0),\dots,f(w_d)$ in $\Cay(\Gamma,S)$ of the same length, since each
consecutive difference lies in $S$.
\end{proof}

Lemma~\ref{lem:nostretch} applies to $\phi$ because every edge difference is
some $\pm g_j\in S$ by construction. Consequently host distances can only be
too small, never too large, and both notions of correctness become one-sided
conditions.

\begin{proposition}\label{prop:shortcut}
Let $f$ be as in Lemma~\ref{lem:nostretch}. Then
\[
  \text{$f$ isometric}\iff\text{no shortcut at any distance},\qquad
  \text{$f$ induced}\iff\text{no shortcut at distance }1,
\]
where a \emph{shortcut} is a pair $u,v$ with
$d_{\Cay(\Gamma,S)}(f(u),f(v))<d_G(u,v)$, and a shortcut at distance $1$ is
one with $d_{\Cay(\Gamma,S)}(f(u),f(v))=1<d_G(u,v)$.
\end{proposition}
\begin{proof}
By Lemma~\ref{lem:nostretch} the host distance never exceeds the graph
distance, so isometry fails exactly when some pair is strictly closer in the
host. For the second equivalence, $f$ is induced exactly when no non-adjacent
pair has adjacent images, that is, exactly when no pair with $d_G\ge2$ has
host distance $1$.
\end{proof}

Everything in the recalled construction --- the oriented partition, the signed
cycle--class matrix, the Smith normal form, $\Gamma_{\mathrm{univ}}$ and the
sublattice folding --- therefore carries over to the induced problem verbatim.
What changes is only the verification step, and the change is the practical
content of this paper: see Section~\ref{sec:complexity}.

\begin{remark}[Folds that isometry rejects]
Since the induced certificate is weaker, folds creating a shortcut at distance
$\ge2$ become admissible. Running the same fold search under the two
certificates gives $6$ against $8$ for $P_5$ and $9$ against $11$ for the
bull, reproducing $\nind$ and $\nu$ respectively in every case tested.
\end{remark}

\section{Complexity}\label{sec:complexity}

\subsection{Space: the host}\label{sec:space}

The host order is the memory footprint of every downstream operation, and the
three schemes differ sharply.

\begin{center}\small
\begin{tabular}{llll}
\toprule
& isometric $\nu$ & induced $\nind$ & completion \\
\midrule
universal bound & $2^{n-1}$ \cite[Cor.~1]{p1} & $(1+o(1))n^2$ \cite{babai-sos1985} & $n$ \cite{comp}\\
tight worst case & odd cycles \cite[Thm.~5]{p2} & $\Theta(n^2)$ (Prop.~\ref{prop:magnitude}) & --- \\
paths & $2(n-1)$ \cite[Cor.~2]{p2} & $n+1$ (Thm.~\ref{thm:path}) & $n$\\
stars $K_{1,q}$ & $2q$ \cite{p2} & $2q$ (Thm.~\ref{thm:bipartite}; \cite{babai-sos1985}) & $n$\\
floor & $\max(n,2\operatorname{diam})$ & $\max(n,2\max_v\alpha(G[N(v)]))$ & $n$\\
host degree & $\le2t$ & $\le2t$, or $2m$ via Sidon & $|S|$\\
\bottomrule
\end{tabular}
\end{center}

The induced scheme is the only one of the three with a \emph{polynomial}
universal bound that requires no search: Corollary~\ref{cor:singer} is a closed
form, and by Proposition~\ref{prop:magnitude} it cannot be improved by more
than a constant factor. Both statements are classical
\cite{babai-sos1985}. The isometric bound is exponential and tight; the
completion host is smallest but is paid for in perturbation.

\subsection{Time: certification dominates, and only for isometry}

Paper \cite{p1} verifies a candidate by comparing all $\binom n2$ distances
against a breadth-first search of the candidate Cayley graph, and
\cite[Rem.~9]{p1} identifies this as the dominant cost, exponential in the
worst case since the binary terminal has order $2^{k}$ with $k$ up to $n-1$.
By Proposition~\ref{prop:shortcut} the induced certificate is a lookup on the
non-edges:
\[
  \text{isometric certificate } \Theta\bigl(N|S| + n^2\bigr),
  \qquad
  \text{induced certificate } \Theta\bigl(n^2\bigr).
\]
The host order $N$ does not appear in the induced cost. We do not claim this
as a complexity separation --- once stated, it is immediate that verifying
adjacency needs only the labelling whereas verifying isometry needs the host
metric. The point is quantitative and concerns the pipeline of \cite{p1}: the
term removed is the only one that can be exponential, and it is the term that
dominates in practice.

\begin{table}[t]
\centering\small
\caption{Measured certificate time, graph fixed at $n=12$, host order varying.
The induced certificate never traverses the host and is flat in $N$. The
isometric column measures \emph{our present implementation}, which runs a
breadth-first search from every host vertex and therefore grows like $N^2$;
see the caveat below. The ratio column compares the two implementations, not
the two algorithms.}
\label{tab:cert}
\begin{tabular}{rrrr}
\toprule
host order $N$ & induced (ms) & isometric (ms) & ratio\\
\midrule
$32$ & $0.032$ & $0.63$ & $20\times$\\
$128$ & $0.028$ & $10.71$ & $389\times$\\
$512$ & $0.031$ & $175.5$ & $5\,610\times$\\
$2048$ & $0.033$ & $2\,974$ & $90\,258\times$\\
$8192$ & $0.030$ & $55\,618$ & $1\,854\,363\times$\\
\bottomrule
\end{tabular}
\end{table}

Table~\ref{tab:cert} and Figure~\ref{fig:complexity}A measure the effect, and
one caveat must be attached to them. The host is a Cayley graph, hence
vertex-transitive, so its distance function is translation-invariant and a
\emph{single} breadth-first search from the identity determines it; that is
what the $\Theta(N|S|+\binom n2)$ figure above assumes. Our implementation
does not exploit this and searches from every vertex, so the measured isometric
times grow like $N^2$: quadrupling $N$ multiplies them by roughly $16$
throughout the table, not by $4$. The ratios in the last column are therefore
the gap between an optimised induced certificate and an unoptimised isometric
one, and they overstate the algorithmic gap by a factor of order $N$. The
qualitative point survives unchanged and is the one we rely on --- verifying
adjacency needs the labelling alone, verifying isometry needs the host metric,
and the term removed is the only one that depends on $N$ --- but the
$1\,854\,363\times$ figure is an artefact of the implementation and should not
be read as an algorithmic separation.

\subsection{Time: the search}

Certification is not the whole cost. Each scheme also searches:

\begin{itemize}[leftmargin=1.4em,itemsep=2pt]
\item \emph{isometric}: over oriented partitions, then over sublattice folds,
each candidate certified as above \cite[Thm.~11]{p1};
\item \emph{induced}: over labellings (or, via
Proposition~\ref{prop:shortcut}, over the same partitions and folds) with the
cheap certificate --- \emph{or no search at all} if the $O(n^2)$ Sidon
construction is accepted;
\item \emph{completion}: over labellings, with the per-labelling optimum
computed exactly in polynomial time by a keep/drop rule subject to a
generation constraint \cite{comp}; the problem is NP-hard for a \emph{fixed
cyclic host} \cite{comp}, while the complexity of $\gamma$ itself, which
minimises over all abelian hosts, is open.
\end{itemize}

The exact optimisation is hard or open in all three cases: for completion,
NP-hard with the host fixed to be cyclic and open in general \cite{comp};
conjectured hard for $\nu$ \cite[Rem.~9]{p1}; and open for $\nind$
(Section~\ref{sec:open}). The three invariants are thus in the same position,
which the table of Section~\ref{sec:space} should be read as recording. What distinguishes the induced scheme is
that a \emph{good enough} answer is free: by Corollary~\ref{cor:singer} no
search at all is needed to reach $(1+o(1))n^2$, and by
Proposition~\ref{prop:magnitude} that is within a factor of two of optimal in
the worst case. All the difficulty is therefore concentrated in the structured
regime, where the optimum is far below $n^2$ and the search is what finds it.

\subsection{What the cheap certificate does not buy}\label{sec:notractable}

It would be easy to read Section~\ref{sec:complexity} as saying that the
induced scheme is computationally easy. It is not, and the distinction matters
enough to state separately.

Three costs must be kept apart. \emph{Verifying} a proposed labelling is
$\Theta(n^2)$ and independent of the host, by
Proposition~\ref{prop:shortcut}; this is the quantity measured in
Table~\ref{tab:cert} and in the last column of Table~\ref{tab:zoo}, and it is
genuinely negligible. \emph{Constructing} a host with no optimality claim is
also cheap: Corollary~\ref{cor:singer} is a closed form requiring no search at
all. \emph{Minimising} the host order is the expensive one, and nothing in
this paper makes it cheap.

The reason the relaxation does not help with minimisation is that it removes
the exponential cost \emph{per candidate}, not the exponential \emph{number of
candidates}. In the quotient route the candidates are oriented partitions of
$E(G)$ together with sublattice folds of $\Gamma_{\mathrm{univ}}$, and their
number grows superexponentially in $|E(G)|$ regardless of how cheaply each one
is checked. In the direct route it is the placements of $n$ vertices into a
group of order $N$, of which there are $\binom Nn n!$ before symmetry
reduction. Proving that no host of order $N$ exists --- which is what an exact
value requires --- means exhausting one of these spaces, and our
implementation quotients only by translation, not by $\Aut(\Gamma)$ or
$\Aut(G)$.

The consequences are visible in our own measurements. The certificate for the
karate club graph would cost microseconds; establishing $\nind$ for it
exhausted $10\,298$~s without reaching either an embedding of order at most
$2n$ or a proof that none exists. The two numbers are not in tension, because
they measure different things. Practically, for a graph of a few hundred or a
few thousand vertices one should not attempt the exact value at all: apply
Corollary~\ref{cor:singer}, which yields a host of order $(1+o(1))n^2$
immediately, and on which the group Fourier transform still costs
$O(N\log N)$, the spectrum being obtainable as the transform of the indicator
function of $S$. The exact invariant $\nind$ is an object of study for small
graphs and a benchmark for heuristics; at those sizes the exact value is not
worth its computational cost, and the closed-form construction should be used
instead.

We record this explicitly, since it is the use we anticipate. The certified
values of Table~\ref{tab:zoo} and of Section~\ref{sec:random}, together with
the exhaustive values for the connected graphs on at most six vertices, are
deposited in \cite{zenodo} as a benchmark set: for any approximate
or anytime method returning an upper bound $\widetilde\eta(G)\ge\nind(G)$, the
ratio $\widetilde\eta/\nind$ can be measured against a known optimum on every
graph in that set. Establishing such a ratio on small graphs is the only
evidence available that a heuristic is trustworthy at the sizes where $\nind$
itself cannot be computed.

\begin{remark}[Host degree]
Two hosts of the same order need not be equally economical downstream, since
the degree $|S|$ governs the cost of building the spectrum of the host. Sidon
hosts have degree $2m$ and are dense; quotient hosts have degree at most $2t$.
For $P_8$ we measured degree $2$ at $N=9$ against degree $14$ at $N=57$, so
the quotient route wins on both order and sparsity.
\end{remark}

\section{Worked examples}\label{sec:examples}

Table~\ref{tab:zoo} gives $\nind$ for every graph worked as an example in
\cite{p1,p2}, together with a broader comparison set. The values are certified
optima, not best-found upper bounds. For each order $N$ in increasing order
from the floor of Corollary~\ref{cor:floor-sharp}, the solver enumerates
\emph{every} abelian group of order $N$ up to isomorphism and performs an
exhaustive backtracking placement in each, pruning on the condition of
Proposition~\ref{prop:2col}; exhausting order $N$ without success therefore
certifies $\nind(G)>N$, and the first success is the optimum. Every returned
labelling is then re-verified independently by a full non-edge check. By
Lemma~\ref{lem:gen} the search may be restricted to hosts generated by their
connection set, which is what makes the enumeration finite in practice. The
solver, the independent re-verification script and every computed value
reported below are deposited in \cite{zenodo}.

\subsection{The double star family}\label{sec:doublestar}

The entry of Table~\ref{tab:zoo} furthest above its floor is the double star
$D_{3,3}$, at $\nind=15$ against a floor of $8$. It is not an anomaly but the
first member of a family, and that family is the one place where we can
currently settle a question rather than pose one.

Write $D_{q,q}$ for the tree obtained from a single edge by attaching $q$
leaves to each endpoint, so $n=2q+2$ and $m=2q+1$. The open neighbourhood of a
centre consists of its $q$ leaves together with the other centre, and these
$q+1$ vertices are pairwise non-adjacent, so $\alpha^*=q+1$ and
Theorem~\ref{thm:floor} gives only $\max(n,2(q+1))=n$. The truth is far
larger.

\begin{table}[h]
\centering\small
\caption{The double star family. Each value is exact: every abelian group of
every order strictly below $\nind$ was searched to completion with no
timeouts, and every returned embedding was re-verified against
Definition~\ref{def:eta}. The last column lists \emph{every} abelian group of
order $\nind$, in invariant-factor form; in each case all of them are hosts.}
\label{tab:doublestar}
\begin{tabular}{@{}crrrrrl@{}}
\toprule
$q$ & $n$ & $m$ & floor & $\nind$ & $\nind-\text{floor}$ & optimal hosts\\
\midrule
$2$ & $6$  & $5$  & $6$  & $10$ & $4$  & $\Z_{10}$ (the only group of order $10$)\\
$3$ & $8$  & $7$  & $8$  & $15$ & $7$  & $\Z_{15}$ (the only group of order $15$)\\
$4$ & $10$ & $9$  & $10$ & $20$ & $10$ & $\Z_{20}$ and $\Z_2\times\Z_{10}$ (both)\\
$5$ & $12$ & $11$ & $12$ & $25$ & $13$ & $\Z_{25}$ and $\Z_5\times\Z_5$ (both)\\
$6$ & $14$ & $13$ & $14$ & $30$ & $16$ & $\Z_{30}$ (the only group of order $30$)\\
\bottomrule
\end{tabular}
\end{table}

\begin{conjecture}\label{conj:doublestar}
$\nind(D_{q,q})=5q$ for all $q\ge2$.
\end{conjecture}

Five values are five values, and we state this as a conjecture rather than a
theorem. But three consequences follow from the \emph{measured} entries of
Table~\ref{tab:doublestar} alone, independently of whether the conjecture
holds.

\paragraph{The constant of Conjecture~\ref{conj:trees} is not $2$.} At $q=5$
we have $\nind(D_{5,5})=25$ against $2n=24$, and at $q=6$, $30$ against $28$.
Both values are exact, so no constant $c\le2$ can bound $\nind(T)/n$ over all
trees. The smallest constant consistent with our data is $c\ge15/7$; if
Conjecture~\ref{conj:doublestar} holds then $\nind(D_{q,q})/n=5q/(2q+2)$
increases to $5/2$ and forces $c\ge5/2$. This settles, negatively, the only
numerical claim we had attached to Conjecture~\ref{conj:trees}; the
conjecture itself --- that some absolute constant exists --- is untouched, and
Problem~\ref{prob:trees} is unchanged.

\paragraph{An explicit family above the census maximum.} The largest ratio in
Table~\ref{tab:zoo} is $2.25$, attained by the single graph
$\mathrm{Frucht}$. Under Conjecture~\ref{conj:doublestar} the double stars
exceed it for every $q\ge10$, and tend to $5/2$. What is now an isolated
extreme row would then be a point on an explicit infinite family.

\paragraph{An extremal test for the floor.} The gap
$\nind(D_{q,q})-\max(n,2\alpha^*)=3q-2$ grows without bound. So no lower bound
depending only on the local independence numbers of neighbourhoods can be
tight on this family, which makes it the natural test case for
Problem~\ref{prob:floor}.

\paragraph{The group type is unconstrained; only the order is.} The last
column of Table~\ref{tab:doublestar} records something we did not expect. At
each $q$ we tested \emph{every} abelian group of order $5q$, and every one of
them is a host: both $\Z_{20}$ and $\Z_2\times\Z_{10}$ at $q=4$, and both
$\Z_{25}$ and $\Z_5\times\Z_5$ at $q=5$ (at $q=3$ and $q=6$ there is only one
group of the relevant order). So on this family the obstruction is entirely
one of \emph{order}: once $|\Gamma|\ge5q$ the group structure imposes no
further condition, and $\nind$ is determined by a counting phenomenon rather
than by which abelian group is available.

This is worth contrasting with the rest of Table~\ref{tab:zoo}, where the
group type does matter --- the Petersen graph needs $\Z_4\times\Z_4$ and no
cyclic host of order $16$ exists for it. The double stars are the opposite
case, and any proof of Conjecture~\ref{conj:doublestar} should therefore be a
counting argument in an arbitrary abelian group of order $5q$, not a
construction inside a particular one.

We also record what \emph{cannot} be read off the column, since it is the
natural first guess. Every host contains an element of order $5$, but for
$q=3,4,6$ that is forced and not observed: $25\nmid5q$, so the $5$-Sylow
subgroup of any abelian group of order $5q$ is $\Z_5$ and splits off. The one
order where the question has content is $q=5$, and there $\Z_{25}$ --- which
has no $\Z_5$ direct factor --- is a host. The data therefore do not support
writing the host as $\Z_5\times A$ with $|A|=q$, and we make no such claim.

\begin{remark}[Symmetry reduction, and the ceiling of
Section~\ref{sec:random}]\label{rem:twins}
The values for $q\ge4$ are out of reach of the search exactly as described in
Section~\ref{sec:notractable}. The $q$ leaves at a centre have identical open
neighbourhoods, so any permutation of them is an automorphism of $D_{q,q}$ and
the search explores $(q!)^2$ relabellings of every embedding. Requiring the
labels within each class of vertices of equal open neighbourhood to increase
is sound --- any embedding may be composed with such an automorphism to put it
in that form, and the composite is still an induced embedding --- and removes
the factor. At $q=5$ and $N=24$ it took two searches that had not terminated
within $900$ seconds down to $6.4$ and $2.8$ seconds.

We draw the methodological consequence rather than leaving it implicit. In
Section~\ref{sec:random} we report a ceiling near $n=12$ and attribute it to
the method. Table~\ref{tab:doublestar} shows that attribution to be too
generous to ourselves: a single quotient by one family of automorphisms moved
$D_{6,6}$ at $n=14$ from unreachable to routine. The ceiling is a property of
what our implementation quotients by, not of exhaustive search, and quotienting
further --- by $\Aut(G)$ in general, and by $\Aut(\Gamma)$ on the host side ---
should move it again.
\end{remark}

\begin{figure}[t]
\centering
\includegraphics[width=\linewidth]{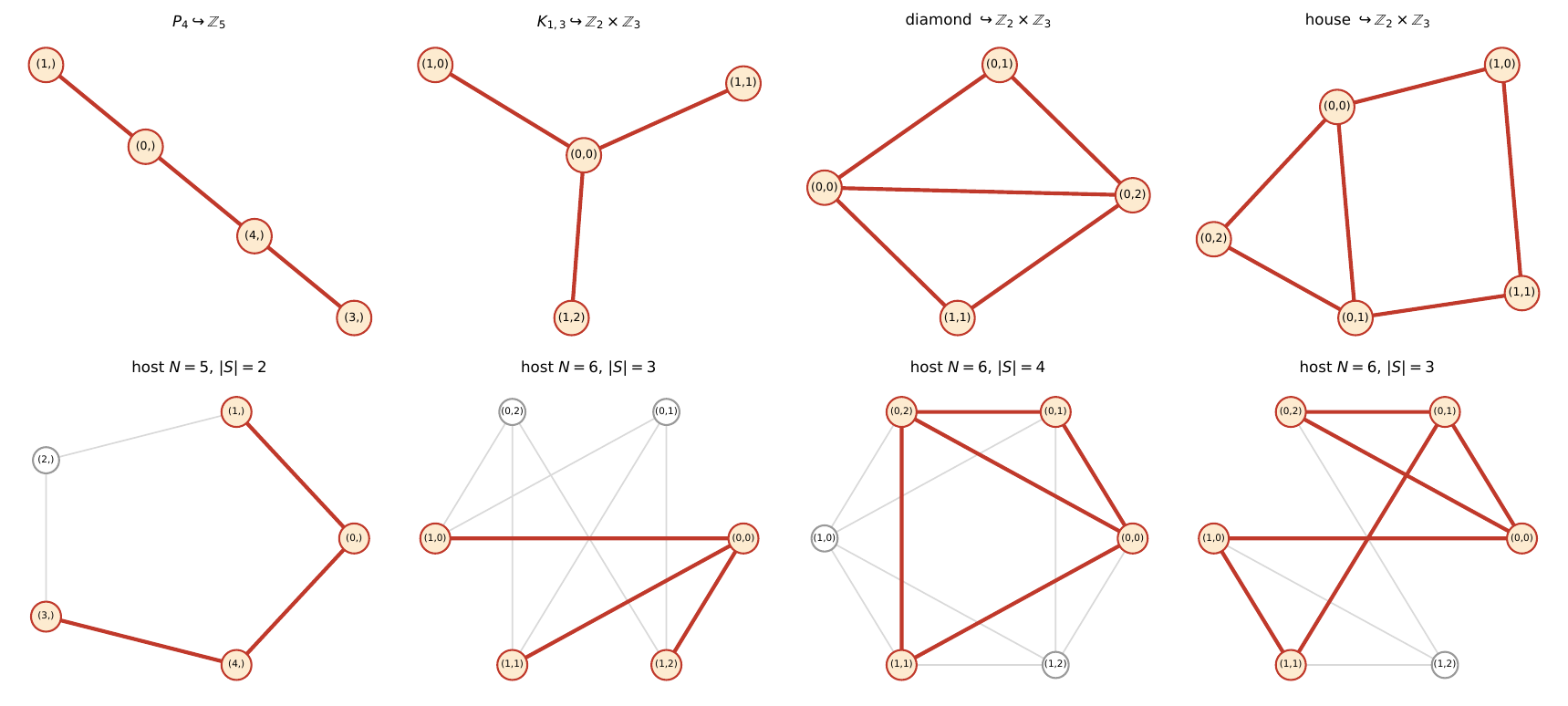}
\caption{Four induced embeddings. Top: the graph, each vertex labelled by its
image in the host group. Bottom: the host Cayley graph, with the image
highlighted and the image of $E(G)$ in red; grey edges are host edges not used
by $G$. No non-edge of $G$ is a host edge, which is the condition of
Proposition~\ref{prop:2col}.}
\label{fig:embed}
\end{figure}

\begin{table}[t]
\centering\small
\caption{Induced host order $\nind$, the floor of Theorem~\ref{thm:floor}, an
optimal host group in \emph{invariant-factor} form
$\Z_{d_1}\times\cdots\times\Z_{d_k}$ with $d_1\mid\cdots\mid d_k$, so that a
single factor means the host is cyclic (the group need not be unique --- the
bull attains $9$ in both $\Z_9$ and $\Z_3\times\Z_3$), the isometric $\nu$
from \cite{p1,p2} where known, and the measured induced certificate time.
Sorted by $\nind/n$.}
\label{tab:zoo}
\begin{tabular}{lrrrrrllr}
\toprule
graph & $n$ & $m$ & floor & $\nind$ & $\nind/n$ & optimal host & $\nu$ & cert.\ (ms)\\
\midrule
$K_3$ & 3 & 3 & 3 & 3 & 1.00 & $\Z_{3}$ & $=3$ & 0.100 \\
$K_{3,3}$ & 6 & 9 & 6 & 6 & 1.00 & $\Z_{6}$ & $=6$ & 0.019 \\
$C_5$ & 5 & 5 & 5 & 5 & 1.00 & $\Z_{5}$ & $=5$ & 0.012 \\
prism $Y_3$ & 6 & 9 & 6 & 6 & 1.00 & $\Z_{6}$ & $=6$ & 0.020 \\
$Q_3$ & 8 & 12 & 8 & 8 & 1.00 & $\Z_{2}\times\Z_{4}$ & $=8$ & 0.033 \\
$K_{2,3}$ & 5 & 6 & 6 & 6 & 1.20 & $\Z_{6}$ & $=6$ & 0.014 \\
house & 5 & 6 & 5 & 6 & 1.20 & $\Z_{6}$ & --- & 0.014 \\
$P_4$ & 4 & 3 & 4 & 5 & 1.25 & $\Z_{5}$ & $=6$ & 0.009 \\
Gr\"otzsch & 11 & 20 & 11 & 16 & 1.45 & $\Z_{4}\times\Z_{4}$ & --- & 0.062 \\
diamond & 4 & 5 & 4 & 6 & 1.50 & $\Z_{6}$ & $=6$ & 0.011 \\
$K_{1,3}$ & 4 & 3 & 6 & 6 & 1.50 & $\Z_{6}$ & $=6$ & 0.010 \\
paw & 4 & 4 & 4 & 6 & 1.50 & $\Z_{6}$ & --- & 0.020 \\
Petersen & 10 & 15 & 10 & 16 & 1.60 & $\Z_{4}\times\Z_{4}$ & $=16$ & 0.051 \\
wheel $W_6$ & 6 & 10 & 6 & 10 & 1.67 & $\Z_{10}$ & --- & 0.023 \\
lollipop $L_{4,2}$ & 6 & 8 & 6 & 10 & 1.67 & $\Z_{10}$ & --- & 0.021 \\
icosahedron & 12 & 30 & 12 & 20 & 1.67 & $\Z_{2}\times\Z_{10}$ & --- & 0.079 \\
friendship $F_3$ & 7 & 9 & 7 & 12 & 1.71 & $\Z_{12}$ & --- & 0.026 \\
bull & 5 & 5 & 5 & 9 & 1.80 & $\Z_{9}$ & $=11$ & 0.015 \\
double star $D_{3,3}$ & 8 & 7 & 8 & 15 & 1.88 & $\Z_{15}$ & --- & 0.034 \\
net & 6 & 6 & 6 & 12 & 2.00 & $\Z_{12}$ & --- & 0.021 \\
Chv\'atal & 12 & 24 & 12 & 24 & 2.00 & $\Z_{24}$ & --- & 0.073 \\
Frucht & 12 & 18 & 12 & 27 & 2.25 & $\Z_{3}\times\Z_{9}$ & --- & 0.071 \\
\bottomrule
\end{tabular}
\end{table}

\begin{figure}[t]
\centering
\includegraphics[width=\linewidth]{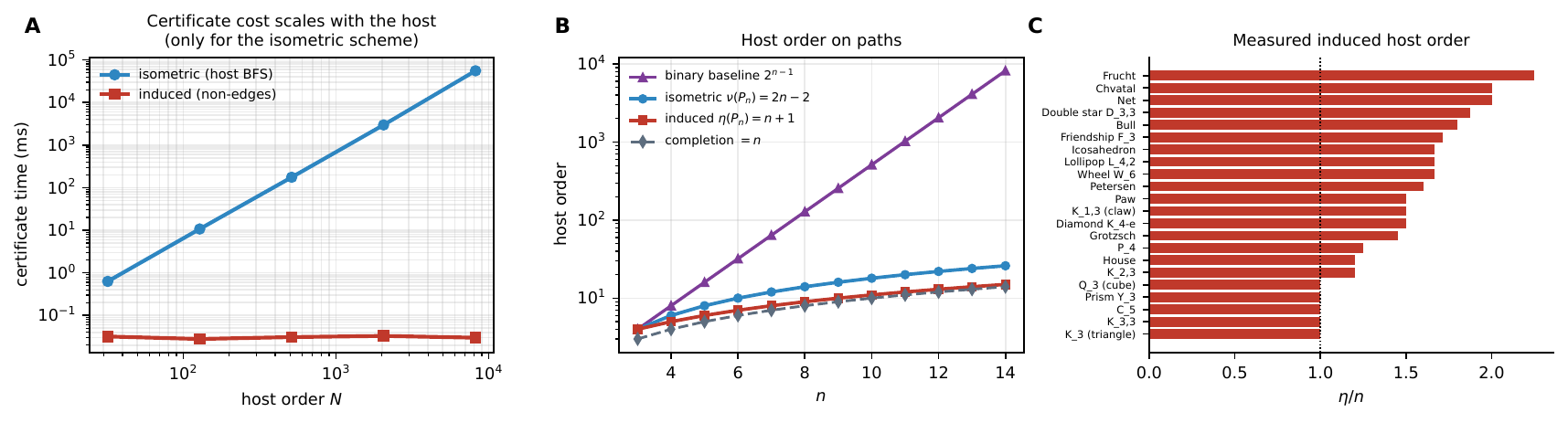}
\caption{\textbf{A} Measured certificate cost against host order
(Table~\ref{tab:cert}). \textbf{B} Host order of the four constructions on
paths. \textbf{C} Measured $\nind/n$ across the examples of
Table~\ref{tab:zoo}.}
\label{fig:complexity}
\end{figure}

\paragraph{The cyclic restriction is cheap on most graphs and expensive on a
few.} Read in invariant-factor form, Table~\ref{tab:zoo} says that
\emph{seventeen} of the $22$ optimal hosts are cyclic. Products of cyclic
groups of coprime orders are cyclic by the Chinese remainder theorem, so
$\Z_2\times\Z_3\cong\Z_6$, $\Z_2\times\Z_5\cong\Z_{10}$,
$\Z_3\times\Z_4\cong\Z_{12}$, $\Z_3\times\Z_5\cong\Z_{15}$ and
$\Z_3\times\Z_8\cong\Z_{24}$ all denote cyclic hosts. Exactly five graphs
have a listed optimal host of rank two --- $Q_3$ ($\Z_2\times\Z_4$),
Gr\"otzsch and Petersen ($\Z_4\times\Z_4$), the icosahedron
($\Z_2\times\Z_{10}$) and Frucht ($\Z_3\times\Z_9$) --- and none has rank
three. Writing the column in primary rather than invariant-factor form
obscures this, which is why we give it in invariant factors.

The consequence for the comparison with $\mathrm{rep}$ is therefore a
dichotomy rather than a uniform verdict. Write $\nind_{\mathrm{cyc}}$ for the
minimum over cyclic hosts alone. On the seventeen graphs with a cyclic optimal
host, $\nind_{\mathrm{cyc}}=\nind$ and the restriction costs nothing. On the
remaining five it can cost a great deal: restricting our own solver to $\Z_N$
returns $36$ for the Petersen graph against the true $16$, and $59$ for Frucht
against $27$ --- factors $2.25$ and $2.19$. For $Q_3$, Gr\"otzsch and the
icosahedron the table exhibits a rank-two optimal host but we have not computed
$\nind_{\mathrm{cyc}}$, so we do not claim a gap there.

The classical per-graph invariant in this area, the representation number
$\mathrm{rep}(G)$ of \cite{erdos-evans1989,evans-isaak-narayan2000}, imposes
the cyclic restriction and one further one, fixing the connection set to be the
units of $\Z_N$. What our measurements show is not that the restriction is
generally harmful --- on this sample it usually is not --- but that its cost is
\emph{concentrated}: harmless on most graphs, and more than a factor two on an
identifiable minority. Characterising that minority is the substantive question
this leaves open, and Problem~\ref{prob:cyclic} states it. All five graphs we
found have a host with a repeated prime among its invariant factors, which is
the natural first guess. This is also the reason the solver of this section
enumerates abelian types rather than $\Z_N$: the enumeration costs little and
is what makes the five cases visible at all.

\begin{corollary}[A free isometric value]\label{cor:nu-petersen}
$\nu(\mathrm{Petersen})=16$.
\end{corollary}
\begin{proof}
$\nind\le\nu$ always, and $\nind(\mathrm{Petersen})=16$ is certified above,
so $\nu\ge16$; \cite[Ex.~2]{p1} exhibits the Clebsch host of order $16$, so
$\nu\le16$.
\end{proof}

\paragraph{Induced rarely beats isometric on structured graphs.} Where $\nu$
is known, $\nind=\nu$ in nine of ten cases; the only strict gain is
$\nind(P_4)=5<6=\nu(P_4)$, an instance of Theorem~\ref{thm:path} against
$\nu(P_k)=2(k-1)$. Notably $\nind(\text{Petersen})=16$ coincides with the
order of the Clebsch host of \cite[Ex.~2]{p1}; that the Petersen graph is an
induced subgraph of the Clebsch graph, a Cayley graph of $\Z_2^4$, is
classical, so the content here is the matching lower bound and the fact that
the optimal group we find, $\Z_4\times\Z_4$, is not isomorphic to $\Z_2^4$.
The reading is that
most of these graphs are vertex-transitive or nearly so, and for them the
isometric host already sits near the injectivity floor, leaving the relaxation
nothing to recover. The relaxation pays where the isometric order is driven up
by the diameter term of \cite[Thm.~2]{p2} --- paths being extreme, with
$\nu=2(n-1)$ growing twice as fast as $\nind=n+1$. Cartesian products divide
along the same line. Products of Cayley graphs are Cayley graphs, so the cube
$Q_3=K_2^{\square3}$ and the prism $Y_3=K_3\square K_2$ sit at $\nind=n$ in
Table~\ref{tab:zoo}, and tori $C_{m}\square C_{m'}$ do too; but path-grids are
irregular, hence strictly above $n$, and Corollary~\ref{cor:grid} places them
just above it, at $n+2\sqrt n+1$ for the square grid. Products are thus a
family where the relaxation is cheap for a structural reason:
Theorem~\ref{thm:product} multiplies hosts, and the factors are already
near-optimal.

\paragraph{Where the cost concentrates.} The largest ratios are Frucht
($2.25$), Chv\'atal ($2.00$), the net ($2.00$) and the double star $D_{3,3}$
($1.88$); by Section~\ref{sec:doublestar} the last of these is the smallest
member of a family whose ratio increases past all of them. Across the $112$ connected graphs on six vertices --- the exhaustive
range of our $\nind$ census, deposited as \cite{zenodo} --- the cost peaks at
intermediate density: mean $\nind/n$ rises from $1.61$ at $m=5$ to $1.80$ at
$m=8$ (density $0.53$) and falls to $1.00$ at $m=15$, where the graph is
complete and hence Cayley. Sparse graphs have few edge-differences to
separate and dense graphs are close to Cayley graphs; the hard case has $\DE$
and $\DN$ both large. This is the finite-$n$ shadow of
Proposition~\ref{prop:magnitude}, whose counting bound is also driven by the
density-$\tfrac12$ regime.

\begin{remark}[The small-$n$ regime is not the general one]
Every value in Table~\ref{tab:zoo} lies below $2.25n$. It would be natural to
read this as evidence for a linear bound, and that reading is wrong: by
Proposition~\ref{prop:magnitude} almost all graphs on $n$ vertices require
$\nind\ge(\tfrac12-o(1))n^2$. The sample here is small and biased towards
vertex-transitive and highly structured graphs, which are exactly the graphs
on which $\nind$ is smallest. What the table measures is the size of the gap
between the structured and the generic regime, not the growth rate of $\nind$.
\end{remark}

\subsection{Measured growth on random graphs}\label{sec:random}

Proposition~\ref{prop:magnitude} is asymptotic, and Table~\ref{tab:zoo}
consists of small, highly structured graphs. To see whether the predicted
growth is visible at all at computationally accessible sizes, we computed
$\nind$ exactly for samples of $G(n,\tfrac12)$ by the procedure of
Section~\ref{sec:examples}: every abelian group of every order from the floor
upward, so that exhausting an order certifies a lower bound.

\begin{center}
\begin{tabular}{@{}ccc@{}}
\toprule
$n$ & exact values of $\nind/n$ & mean\\
\midrule
$6$  & $1.33,\ 1.50,\ 1.67$ & $1.50$\\
$8$  & $1.75,\ 1.88,\ 2.12$ & $1.92$\\
$10$ & $2.30$               & $2.30$\\
\bottomrule
\end{tabular}
\end{center}

The mean rises monotonically, which is what Proposition~\ref{prop:magnitude}
requires. Three points do not establish a growth rate; they establish only
that $\nind/n$ is not constant on random graphs over the range where it can be
certified.

Two negative findings are worth recording, because they bound what experiments
of this kind can show.

First, \emph{we could not separate sparse from dense families at accessible
sizes}. At $n=9$ the exact values are $1.56$ for a random tree, $2.11$ for a
Barab\'asi--Albert graph, $2.22$ for a Watts--Strogatz graph and $2.33$ for a
random geometric graph, against a mean of $1.92$ for $G(8,\tfrac12)$: the
sparse families sit \emph{above} the dense one. The explanation is that
sparsity is asymptotic. At $n=9$ these generators produce $14$ to $18$ edges,
essentially the density of $G(9,\tfrac12)$, so at this size there is no sparse
regime to observe. Only trees, whose $n-1$ edges make them sparse at every
size, separate clearly, and they are the lowest entry at every $n$ we
resolved.

Second, \emph{the exhaustive method has a hard ceiling near $n=12$}. Median
running time per graph rose from under a second at $n=6$ to $2.4$~s at $n=8$,
$271$~s at $n=10$, $373$~s at $n=12$ and beyond $1600$~s at $n=16$, at which
point the search returns neither an exact value nor a certified bound. The
extreme case is Zachary's karate club, $n=34$: after $10\,298$~s the search had
neither found an embedding of order at most $2n=68$ nor exhausted the
possibilities, so nothing at all is certified about it beyond the trivial
$34\le\nind\le1407$ from injectivity and Corollary~\ref{cor:singer}.
Determining the constant of Proposition~\ref{prop:magnitude} empirically is
therefore out of reach by this route. We are careful about what "this route"
means: Remark~\ref{rem:twins} shows that quotienting by a single family of
graph automorphisms moved $D_{6,6}$ at $n=14$ from unreachable to seconds, so
the ceiling reported here is a property of the symmetries our implementation
exploits --- translation only --- and not of exhaustive search as such. On the
random graphs of this section $\Aut(G)$ is typically trivial and no such
saving is available, which is why the ceiling bites here and not on the double
stars.
A third observation is positive, and concerns stability rather than growth.
Two Watts--Strogatz graphs on $n=9$ vertices with $m=18$ edges each --- the
same size, the same generator, the same rewiring probability --- gave
$\nind=9$ and $\nind=20$, both exact. The first had survived rewiring as a
circulant, so Proposition~\ref{prop:floor-n} applies and the host is the graph
itself; in the second, rewiring a couple of edges raised the minimum host order
by a factor of $2.2$. So $\nind$ is not stable under small perturbations of the
graph, which is the induced counterpart of the sensitivity that motivates the
completion invariant $\gamma$ of \cite{comp}, and a caution against reading
family averages as if $\nind$ were a smooth function of $n$ and $m$.

Finally, some graphs are resolved only from below. For a second
Barab\'asi--Albert graph on $9$ vertices --- not the one tabulated above, which
was resolved exactly --- the search exhausted every abelian group of every
order up to $18$ without success, certifying $\nind>2n$; this is a proof, not a
timeout, and such lower bounds are reported as strict inequalities throughout.

\section{Open problems}\label{sec:open}

We first record what is \emph{not} open. By Proposition~\ref{prop:magnitude},
$\nind(G)=O(n)$ is false, and no search-free construction can attain $O(n)$
for all graphs; both follow from \cite[Prop.~2.3]{babai-sos1985} together with
Corollary~\ref{cor:singer}. The worst-case order of magnitude, $\Theta(n^2)$,
has been settled since 1985. What remains open is the behaviour on restricted
classes, the constant, and the computational question.

\begin{problem}[Trees; cf.\ Conjecture~\ref{conj:trees}]\label{prob:trees}
Is $\nind(T)=O(n)$ for every tree $T$?
\end{problem}
The counting bound of Proposition~\ref{prop:magnitude} is vacuous here, since
there are only $n^{n-2}$ labelled trees. Babai and S\'os prove
$\nind(T)\le n^2$ and explicitly ask this question, noting that they have no
non-linear lower bound \cite[Thm.~7.1 and Rem.~7.2]{babai-sos1985}. Our double stars
supply the best lower bound on the constant we know: $c\ge15/7$ from measured
values, and $c\ge5/2$ under Conjecture~\ref{conj:doublestar}.
Proposition~\ref{prop:biparred} reduces the parity route to an additive
question about cross-differences, and caterpillars are the natural first case.
We regard this as the sharpest question in the area and the one our machinery
is best placed to attack.

\begin{problem}[Cartesian products]\label{prob:products}
Is $\nind(G_1\square G_2)$ ever much smaller than the bound of
Theorem~\ref{thm:product}, and does equality characterise any natural class?
\end{problem}
Corollary~\ref{cor:grid} gives $\nind(P_m\square P_m)=(1+o(1))n$, and the
smallest case already shows the bound is not exact:
$\nind(P_3\square P_3)=15$ against the product bound $16$.

\begin{problem}[Sparse classes]\label{prob:sparse}
Determine the growth of $\nind$ on graphs of bounded degree, bounded
degeneracy, or planar graphs.
\end{problem}
Here the counting argument does give something --- roughly
$\Omega(\Delta n\log n/\log(\Delta n))$ for maximum degree $\Delta$ --- so the
target is a matching upper bound rather than linearity.

\begin{problem}[The constant]\label{prob:constant}
For $G(n,\tfrac12)$, Proposition~\ref{prop:magnitude} confines the constant to
$[\tfrac12,1]$. Which endpoint is correct?
\end{problem}
Equivalently: is the Singer construction asymptotically optimal, or does the
freedom to let differences collide within $\DE$ and within $\DN$ buy a factor
of two?

\begin{problem}[Complexity]\label{prob:complexity}
Is computing $\nind$ NP-hard? Is $\nind(G)\le N$ hard for $N$ given in unary?
\end{problem}
We expect NP-hardness by analogy with the fixed-cyclic-host case of $\gamma$
\cite{comp}, but have
no reduction, and Section~\ref{sec:notractable} explains why none of our
machinery yields one: the relaxation cheapens each candidate without reducing
how many there are. The decision problem $\nind(G)=n$ asks whether $G$ is an
abelian Cayley graph, equivalently whether $\Aut(G)$ contains a regular
abelian subgroup, and is the natural first case. Our experiments suggest the
difficulty concentrates just below the optimum, where instances are barely
infeasible and exhaustion is forced --- the behaviour of a phase transition
rather than of uniform hardness.

There is a dichotomy here worth stating, because it runs in our favour. If the
host is required to be cyclic, deciding $\nind_{\mathrm{cyc}}(G)=n$ is the
problem of recognising a circulant, and that is solvable in polynomial time by
Evdokimov and Ponomarenko \cite{evdokimov-ponomarenko2003}; the companion
isomorphism problem for circulants is settled by Muzychuk's proof of
\'Ad\'am's conjecture in the square-free case and its extension
\cite{muzychuk1995,muzychuk1997}. So on cyclic hosts the decision problem at
the bottom of the range is easy while the minimisation is conjectured hard.
Over all abelian groups even the decision problem is open, and it is that gap
that Problem~\ref{prob:complexity} asks about.

\begin{problem}[The cyclic gap]\label{prob:cyclic}
Write $\nind_{\mathrm{cyc}}$ for the minimum over cyclic hosts only. Is
$\nind_{\mathrm{cyc}}/\nind$ bounded?
\end{problem}
On $17$ of our $22$ graphs the ratio is $1$; we measured $2.25$ for the
Petersen graph and $2.19$ for Frucht. The sharper question is therefore not
whether the ratio is bounded but which graphs have ratio greater than $1$ at
all. Every graph we found with a strictly larger cyclic optimum has an optimal
host whose invariant factors share a prime; whether that is necessary,
sufficient, or neither, we do not know.

\begin{problem}[Sharpening the floor]\label{prob:floor}
Find a lower bound on $\nind$ sensitive to more than one neighbourhood at a
time.
\end{problem}
The local independence term of Theorem~\ref{thm:floor} is tight on complete
bipartite graphs and loose on the Petersen graph, where it gives $6$ against
the true $16$. The double stars are worse still and are the natural target: by
Section~\ref{sec:doublestar} the gap there is $3q-2$, unbounded, so no
neighbourhood-local bound can be tight on that family. Unbalanced complete multipartite graphs
(Remark~\ref{rem:multipartite}) are a second case where our construction
overshoots the floor and neither bound is known to be tight.

\section*{Data availability and reproducibility}
The exact values reported in Section~\ref{sec:examples}, the underlying search
code, and the certificates are deposited at
\texttt{doi:10.5281/zenodo.21944455} \cite{zenodo}. Every value of $\nind$
stated in this paper is machine-certified in both directions: an explicit
labelling $f:V(G)\to\Gamma$ together with a connection set $S$ witnesses the
upper bound, and exhaustion over all abelian groups of smaller order, with all
symmetric connection sets up to the normalisation of Lemma~\ref{lem:gen},
witnesses the lower bound. The census is reproducible from the deposit with a
single command; graph labels follow Read and Wilson \cite{read-wilson1998}.

\section*{Declaration on the use of artificial intelligence}
The authors declare that the artificial-intelligence assistant Claude
(Anthropic) was used during the preparation of this manuscript, in two roles:
support with the implementation, debugging and reproducibility of the search
and certification software; and language and editing support in drafting. All
definitions, theorems, proofs and their verification, together with the
conception, scientific direction and conclusions of this work, are the
authors' own. The authors have reviewed the entire manuscript and take full
responsibility for its content.

\end{document}